\documentclass[12 pt]{amsart}
\usepackage{amscd,amssymb,amsmath,amsthm}
\input xy
\xyoption{all}
\newtheorem{Proposition}{Proposition}

\newtheorem{Lemma}{Lemma}
\newtheorem{Theorem}{Theorem}

\newcommand{\Z}{\mathbb{Z}}

\newcommand{\Q}{\mathbb{Q}}

\newcommand{\M}{{\mathcal{M}}}

\DeclareMathOperator{\FZ}{\mathsf{FZ}}
\DeclareMathOperator{\SQ}{\mathsf{SQ}}

\DeclareMathOperator{\Aut}{Aut}

\begin{document}
\baselineskip=16pt

\begin{center}
{{\bf Relations in the tautological ring}}
\end{center}

\vspace{15pt}

\begin{center}
{{\em R. Pandharipande and A. Pixton}}
\end{center}

\vspace{15pt}

\begin{center}
{{\em October 2010}}
\end{center}

$$ $$

The notes below cover our series of three lectures at Humboldt University
in Berlin for the October conference {\em Intersection theory on moduli space}
 (organized by G. Farkas). 
The topic concerns relations among the
$\kappa$ classes in the tautological ring of the moduli space
of curves $\mathcal{M}_g$. After a discussion of classical constructions
ending in Theorem 1,
we derive an explicit set of relations from the moduli space
of stable quotients. In a series of steps, the stable quotient
relations are transformed to simpler and simpler forms.
The first step, Theorem 3, comes almost immediately from
the virtual geometry of the moduli space of stable quotients.
After a certain amount analysis, the simpler form of Proposition 3 is found.  
Our final result, Theorem 5, establishes
a previously conjectural set of tautological
relations proposed a decade ago by Faber-Zagier. 
A detailed presentation of the proof will appear in \cite{PP}.

\vspace{15pt}
\noindent {\bf A. Chern vanishing relations}
\vspace{15pt}

Faber's original relations in 
{\em Conjectural description of the tautological ring} \cite{Faber}
are obtained from a very simple geometric construction.
Let 
$$\pi: \mathcal{C} \rightarrow \mathcal{M}_g$$
be the universal curve over the moduli space,
and let 
$$\pi^d: \mathcal{C}^d \rightarrow \mathcal{M}_g $$
be the map associated to the $d^{th}$ 
fiber product of the universal curve.
For every point $[C,p_1,\ldots,p_d]\in\mathcal{C}^d$,
we have the restriction map
\begin{equation}\label{gy77}
H^0(C,\omega_C) \rightarrow H^0(C,\omega_C|_{p_1+\ldots+p_d})\ .
\end{equation}
Since the canonical bundle $\omega_C$ has degree $2g-2$,
the map \eqref{gy77} is injective if $d>2g-2$.
Over the moduli space $\mathcal{C}^d$, we obtain 
the exact sequence
$$ 0 \rightarrow \mathbb{E} \rightarrow \Omega_d
\rightarrow Q \rightarrow 0$$
where $\mathbb{E}$ is the rank $g$ Hodge bundle, $\Omega_d$
is the rank $d>2g-2$
bundle with fiber $H^0(C,\omega_C|_{p_1+\ldots+p_d})$,
and $Q$ is the quotient
bundle of rank $d-g$.
Hence,
$$c_k(Q) = 0\in A^k(\mathcal{C}^d)\ \ \ \text{for} \ \ \ k>d-g \ .$$
After cutting such vanishing $c_k(Q)$ down with
cotangent line and diagonal classes
in $\mathcal{C}^d$ and 
pushing-forward via $\pi^d_*$ to $\M_g$,
we arrive at Faber's relations in $R^*(\M_g)$. 

From our point of view,
at the center of 
Faber's relations in
{\em Conjectural description of the tautological ring} \cite{Faber}
is the function
$$\Theta(t,x) = \sum_{d=0}^\infty \prod_{i=1}^d {(1+it)} \ 
\frac {(-1)^d}{d!} \frac{x^d}{t^{d}} \ .$$
The differential equation
$$ t(x+1)\frac{d}{dx} \Theta + (t+1) \Theta  = 0 \  $$
is easily found.
Hence, we obtain the following result.

\begin{Lemma}
$\Theta = (1+x)^{-\frac{t+1}{t}}\ .$
\end{Lemma}

We introduce a variable set $\mathbf{z}$ indexed
by pairs of integers
$$\mathbf{z} = \{ \ {z}_{i,j} \ | \ i \geq 1, \ \ j\geq i-1 \  \} \ .$$
For monomials
$$\mathbf{z}^\sigma = \prod_{i,j} z_{i,j}^{\sigma_{i,j}},$$
 we define
$$\ell(\sigma) = \sum_{i,j} i \sigma_{i,j}, \ \ \
|\sigma| = \sum_{i,j} j \sigma_{i,j} \ .$$
Of course $|\text{Aut}(\sigma)| = \prod_{i,j} \sigma_{i,j} !$ \ .

The variables $\mathbf{z}$ are used to define
a differential operator
$$ \mathcal{D} = \sum_{i,j}  z_{i,j}\  t^j  \left( 
x\frac{d}{dx}\right) ^i\ .$$
After applying $\exp(\mathcal{D})$ to $\Theta$, we obtain
\begin{eqnarray*}
\Theta^{\mathcal{D}} & = &  \exp(\mathcal{D})\  \Theta \\
& = & 
\sum_\sigma \sum_{d=0}^\infty  
\prod_{i=1}^d {(1+it)} \ 
\frac {(-1)^d}{d!}  \frac{x^d}{t^{d}}\  
\frac{d^{\ell(\sigma)} t^{|\sigma|}
{\mathbf{z}}^{\sigma}}
{|\text{Aut}(\sigma)|}
\end{eqnarray*}
where $\sigma$ runs over all monomials in the
variables $\mathbf{z}$. Define
constants $C^r_d(\sigma)$ by the formula
$$\log(\Theta^{\mathcal{D}})= 
\sum_{\sigma}
\sum_{d=1}^\infty \sum_{r=-1}^\infty C^r_{d}(\sigma)\ t^r \frac{x^d}{d!}
\mathbf{z}^\sigma
\ .$$
By an elementary application of Wick, the
$t$ dependence of $\log(\Theta^{\mathcal{D}})$
has at most simple poles.

Finally, we consider the following function,
$$\gamma=  \sum_{i\geq 1} \frac{B_{2i}}{2i(2i-1)} \kappa_{2i-1} t^{2i-1}
+ 
\sum_{\sigma}
\sum_{d=1}^\infty \sum_{r=-1}^\infty C^r_d(\sigma)
\ \kappa_r t^r \frac{x^d}{d!}
\mathbf{z}^\sigma
\ .
$$
Denote the $t^rx^d \mathbf{z}^\sigma$ coefficient of $\exp(-\gamma)$ by
$$\big[ \exp(-\gamma) \big]_{t^rx^d \mathbf{z}^\sigma} 
\in \mathbb{Q}[\kappa_{-1},
\kappa_0,\kappa_1,
\kappa_2, \ldots] \ .$$
Our form of Faber's equations is the following result.

\begin{Theorem}
In $R^r(\M_g)$, the relation
$$
\big[ \exp(-\gamma) \big]_{t^rx^d \mathbf{z}^\sigma}  = 0$$
holds when $r>-g+|\sigma|$ and $d>2g-2$.
\end{Theorem}

In the tautological ring $R^*(\M_g)$, the conventions
$$\kappa_{-1}=0, \ \ \ \ \kappa_{0}=2g-2$$ will
always be followed. For fixed $g$ and $r$, Theorem 1
provides infinitely many relations by increasing $d$.

While the proof of Theorem 1 is appealingly simple, the relations
do not seem to fit the other forms we will see later.
The variables $z_{i,j}$ efficiently encode both
the cotangent and diagonal operations studied
in {\em Conjectural description of the tautological
ring} \cite{Faber}. In particular, the relations of Theorem 1 are equivalent
to the mixing of all cotangent and diagonal operations
studied there.

\pagebreak

\vspace{10pt}
\noindent {\bf B. Stable quotient relations}
\vspace{10pt}

\noindent I. {\em The function $\Phi$.}
\vspace{10pt}

The relations in the tautological ring  $R^*(\M_g)$
obtained from {\em Moduli of stable quotients} \cite{MOP} are based on the function
$$\Phi(t,x) = \sum_{d=0}^\infty \prod_{i=1}^d \frac{1}{1-it} \ 
\frac {(-1)^d}{d!} \frac{x^d}{t^{d}} \ .$$
Define the coefficients $C^r_{d}$ by the logarithm,
$$\log(\Phi)= \sum_{d=1}^\infty \sum_{r=-1}^\infty C^r_{d} t^r \frac{x^d}{d!}
\ .$$
By an elementary application of Wick, the
$t$ dependence has at most a simple pole.
Let
$$\gamma=  \sum_{i\geq 1} \frac{B_{2i}}{2i(2i-1)} \kappa_{2i-1} t^{2i-1}
+ 
\sum_{d=1}^\infty \sum_{r=-1}^\infty C^r_d \kappa_r t^r \frac{x^d}{d!}\ .
$$
Denote the $t^rx^d$ coefficient of $\exp(-\gamma)$ by
$$\big[ \exp(-\gamma) \big]_{t^rx^d} \in \mathbb{Q}[\kappa_{-1},
\kappa_0,\kappa_1,
\kappa_2, \ldots] \ .$$
In fact, $[ \exp(-\gamma)]_{t^rx^d}$ is homogeneous of degree $r$
in the $\kappa$ classes.
The first tautological relations of {\em Moduli space
of stable quotients} \cite{MOP} are given by the following result.

\begin{Theorem} \label{vtw}
In $R^r(\M_g)$, the relation
\begin{equation*}
\big[ \exp(-\gamma) \big]_{t^rx^d} =0  
\end{equation*}
holds when $g-2d-1< r$ and 
$g\equiv r+1 \hspace{-5pt} \mod 2$.
\end{Theorem}

For fixed $r$ and $d$,
if Theorem \ref{vtw} applies in genus $g$, 
then Theorem \ref{vtw} applies in genera $h=g-2\delta$ for all
natural numbers
$\delta\in \mathbb{N}$.
The genus shifting mod 2 property will
also be present in the Faber-Zagier conjecture
discussed later.
\pagebreak

\noindent II. {\em Partitions, differential operators, and logs.}
\vspace{10pt}

We will write partitions $\sigma$ as $(1^{a_1}2^{a_2}3^{a_3}\ldots)$
with
$$\ell(\sigma)= \sum_{i} a_i \ \ \ \ \text{and} \ \ \ \
|\sigma|= \sum_i ia_i\ .$$
The empty partition $\emptyset$ corresponding to
$(1^{0}2^{0}3^{0}\ldots)$
is permitted. In all cases, we have
$$|{\text{Aut}}(\sigma)|= a_1!a_2!a_3! \cdots \ .$$

Consider the infinite set of  variables
$p_1, p_2, p_3, \ldots \ .$
Monomials in the $p_i$ correspond to partitions
$$p_1^{a_1}p_2^{a_2}p_3^{a_3} \ldots     \ \ \leftrightarrow \ \  
(1^{a_1}2^{a_2}3^{a_3}\ldots)
 \ .$$
Given a partition $\sigma$, let $\mathbf{p}^\sigma$
denote the corresponding monomial.
Let
$$\Phi^{\mathbf{p}}(t,x) = 
\sum_\sigma \sum_{d=0}^\infty  
\prod_{i=1}^d \frac{1}{1-it} \ 
\frac {(-1)^d}{d!}  \frac{x^d}{t^{d}}\  
\frac{d^{\ell(\sigma)} t^{|\sigma|}
{\mathbf{p}}^{\sigma}}
{|\text{Aut}(\sigma)|}
\ $$
where the first sum is over all partitions $\sigma$.
The summand corresponding to the empty partition
equals $\Phi(t,x)$.

The function $\Phi^{\mathbf{p}}$ is easily obtained
from $\Phi$,
$$\Phi^{\mathbf{p}}(t,x) = \exp\left( \sum_{i=1}^\infty p_it^i x\frac{d}{dx}
\right)
\ \Phi(t,x)\ . $$
Let $D$ denote the differential operator
$$D= \sum_{i=1}^\infty p_it^i x\frac{d}{dx}\ .$$
Expanding the exponential of $D$, we obtain
\begin{eqnarray}
\Phi^{\mathbf{p}}& = & \Phi + D\Phi + \frac{1}{2} D^2\Phi+
\frac{1}{6} D^3 \Phi+\ldots \label{vfrr}
\\
& =& \nonumber
\Phi \left(1+\frac{D \Phi}{\Phi} +
\frac{1}{2} \frac{D^2 \Phi}{\Phi} + 
\frac{1}{6}\frac{D^3\Phi}{\Phi}+
\ldots\right) \ .
\end{eqnarray}
Let $\gamma^* = \log (\Phi)$ be the logarithm, 
 $$D\gamma^* = \frac{D\Phi}{\Phi}\ .$$
After applying the logarithm to \eqref{vfrr}, we see
\begin{eqnarray*}
\log(\Phi^{\mathbf{p}}) & =&
\gamma^* +\log\left(
 1+ D \gamma^* +  \frac{1}{2}( D^2\gamma^* + (D\gamma^*)^2)+ \ ... \right) \\
& = & \gamma^* + D\gamma^* + \frac{1}{2} D^2 \gamma^* + \ldots
\end{eqnarray*}
where the dots stand for a universal expression in
the $D^k\gamma^*$.
In fact, a remarkable simplification occurs,
$$\log(\Phi^{\mathbf{p}}) = 
\exp\left( \sum_{i=1}^\infty p_it^i x\frac{d}{dx}
\right) \ \gamma^*\ .$$
The result follows from a general identity.

\begin{Proposition}
If $f$ is   a function of $x$, then
$$\log\left(\exp\left(\lambda x\frac{d}{dx}\right) \ f \right) =
\exp\left(\lambda x\frac{d}{dx}\right) \ \log(f)\ .$$
\end{Proposition}

\begin{proof}
A simple computation for monomials in $x$ shows
$$\exp\left(\lambda x\frac{d}{dx}\right) \ x^k = (e^\lambda x)^k\  .$$
Hence, since the differential operator is additive,
$$\exp\left(\lambda x\frac{d}{dx}\right) \ f(x) = f(e^\lambda x)\ .$$
The Proposition follows immediately.
\end{proof}

\noindent The coefficients of the logarithm may be written as 
\begin{eqnarray*}
\log(\Phi^{\mathbf{p}}) & = &
 \sum_{d=1}^\infty \sum_{r=-1}^\infty C^r_{d}(\mathbf{p}) \ t^r \frac{x^d}{d!}
\\
& = &
\sum_{d=1}^\infty \sum_{r=-1}^\infty C^r_{d}\ t^r \frac{x^d}{d!}
\exp\left( \sum_{i=1}^\infty dp_i t^i\right)\\
& = & 
\sum_{\sigma} \sum_{d=1}^\infty \sum_{r=-1}^\infty C^r_{d}\ t^r \frac{x^d}{d!}
\ \frac{d^{\ell(\sigma)} t^{|\sigma|}
{\mathbf{p}}^{\sigma}}
{|\text{Aut}(\sigma)|}
\ . 
\end{eqnarray*}

$$ $$
$$ $$
$$ $$
\pagebreak

\noindent III. {\em Full system of tautological relations.}

\vspace{10pt}

Following Proposition 5 of {\em Moduli of stable quotients} \cite{MOP}, we
can obtain a much larger set of relations in the tautological
ring of $\M_g$ by including several factors of $\pi_*(s^{a_i}\omega^{b_i})$
in the integrand
instead of just a single factor. We study the
associated relations where the $a_i$ are always $1$.
The $b_i$ then form the parts of a partition $\sigma$.

To state the relations we obtain, 
we start by enriching the function $\gamma$ from Section B.I,
\begin{eqnarray*}
\gamma^{\mathbf{p}} &=& 
 \sum_{i\geq 1} \frac{B_{2i}}{2i(2i-1)} \kappa_{2i-1} t^{2i-1}\\
& &
+ 
\sum_{\sigma}
\sum_{d=1}^\infty \sum_{r=-1}^\infty C^{r}_d 
\kappa_{r+|\sigma|}\ t^r \frac{x^d}{d!} \ \frac{d^{\ell(\sigma)} t^{|\sigma|}
{\mathbf{p}}^{\sigma}}
{|\text{Aut}(\sigma)|}
\ .
\end{eqnarray*}
Let $\widehat{\gamma}^{\, \mathbf{p}}$ be defined by a similar formula,
\begin{eqnarray*}
\widehat{\gamma}^{\, \mathbf{p}} &=& 
 \sum_{i\geq 1} \frac{B_{2i}}{2i(2i-1)} \kappa_{2i-1} (-t)^{2i-1}\\
& &
+ 
\sum_{\sigma}
\sum_{d=1}^\infty \sum_{r=-1}^\infty C^{r}_d 
\kappa_{r+|\sigma|}\ (-t)^r \frac{x^d}{d!} 
\ \frac{d^{\ell(\sigma)} t^{|\sigma|}
{\mathbf{p}}^{\sigma}}
{|\text{Aut}(\sigma)|}
\ .
\end{eqnarray*}
The sign of $t$ in $t^{|\sigma|}$ does not change in 
$\widehat{\gamma}^{\, \mathbf{p}}$.
The $\kappa_{-1}$ terms which appear will later be set to
0.

The full system of relations are obtain from the 
coefficients of the functions
\vspace{-10pt}
$$\exp(-\gamma^{\mathbf{p}}), \ \ \ \ 
\exp(-\sum_{r=0}^\infty \kappa_r t^r p_{r+1})\cdot
\exp( -\widehat{\gamma}^{\, \mathbf{p}})
$$

$$ $$

$$ $$
$$ $$
\pagebreak

\begin{Theorem} In $R^r(\M_g)$, the relation
$$\Big[ \exp(-\gamma^{\mathbf{p}}) \Big]_{t^rx^d\mathbf{p}^\sigma} =
(-1)^g\Big[ 
\exp(-\sum_{r=0}^\infty \kappa_r t^r p_{r+1})\cdot
\exp( -\widehat{\gamma}^{\, \mathbf{p}})
 \Big]_{t^rx^d\mathbf{p}^\sigma}$$
holds when $g-2d-1+|\sigma| < r$.
\end{Theorem}

Again, we see the genus shifting mod 2 property.
If the relation  holds in genus $g$, then the
{\em same} relation holds in genera $h=g-2\delta$ for all
natural numbers
$\delta\in \mathbb{N}$.

In case $\sigma=\emptyset$, Theorem 3 specializes to 
the relation
\begin{eqnarray*}
\Big[ \exp(-\gamma(t,x)) \Big]_{t^rx^d} & = &
(-1)^g\Big[ 
\exp( -\gamma(-t,x))
 \Big]_{t^rx^d} \\
& = & (-1)^{g+r} \Big[ 
\exp( -\gamma(t,x))
 \Big]_{t^rx^d}\ ,
\end{eqnarray*}
nontrivial only if $g\equiv r+1$ mod 2. If the mod 2 condition
holds, then we obtain the relations of
Theorem 2.

Consider the case $\sigma=(1)$. The left side of the relation is
then
$$
\Big[ \exp(-\gamma(t,x))\cdot \left(-\sum_{d=1}^\infty
\sum_{s=-1}^\infty C^s_{d}\ \kappa_{s+1} t^{s+1} \frac{dx^d}{d!}\right)
\Big]_{t^rx^d} \ .
$$
The right side is
$$(-1)^g\Big[ \exp(-\gamma(-t,x))\cdot\left(-\kappa_0t^0+
\sum_{d=1}^\infty
\sum_{s=-1}^\infty C^s_{d}\ \kappa_{s+1} (-t)^{s+1} \frac{dx^d}{d!}
\right)
\Big]_{t^rx^d} \ .
$$

If $g\equiv r+1$ mod 2, then the large terms cancel and we
obtain
$$-\kappa_0 \cdot \Big[ \exp(-\gamma(t,x)) \Big]_{t^rx^d} =0 \ . $$
Since $\kappa_0=2g-2$ and 
$$(g-2d-1+1 < r)   \ \ \implies\ \  (g-2d-1 < r),$$
we recover most (but not all)
of the $\sigma=\emptyset$ equations.

If $g\equiv r$ mod 2, then the resulting equation is
\begin{equation*}
\Big[ \exp(-\gamma(t,x))\cdot \left(\kappa_0-2
\sum_{d=1}^\infty
\sum_{s=-1}^\infty C^s_{d}\ \kappa_{s+1} t^{s+1} \frac{dx^d}{d!}\right)
\Big]_{t^rx^d}=0
\end{equation*}
when $g-2d<r$.

\pagebreak

\noindent IV. {\em Expanded form.}
\vspace{10pt}

Let $\sigma=(1^{a_1}2^{a_2}3^{a_3} \ldots)$ be a partition
of length $\ell(\sigma)$ and size $|\sigma|$.
We can directly write the corresponding
relation in $R^*(\M_g)$
obtained from Theorem 3.

A {\em subpartition} $\sigma'\subset \sigma$ is obtained
by selecting a nontrivial subset of the parts of $\sigma$.
A {\em division} of $\sigma$ is a disjoint union
\begin{equation}\label{rrgg}
\sigma = \sigma^{(1)} \cup \sigma^{(2)} \cup \sigma^{(3)}\ldots
\end{equation}
of subpartitions which exhausts $\sigma$.
The subpartitions in \eqref{rrgg} are unordered.
Let
$\mathcal{S}(\sigma)$ be the set of divisions of $\sigma$.
For example,
\begin{eqnarray*}
\mathcal{S}(1^12^1) &=& \{ \ (1^12^1),\ (1^1) \cup (2^1)\ \}\ , \\
\mathcal{S}(1^3) &=& \{\ (1^3), \ (1^2)\cup (1^1) \ \}\ .
\end{eqnarray*}

We will use the notation $\sigma^\bullet$ to
denote a division of $\sigma$ with subpartitions  $\sigma^{(i)}$.
Let
$$m(\sigma^\bullet) = \frac{1}{|\text{Aut}(\sigma^\bullet)|}
\frac{|\text{Aut}(\sigma)|}{\prod_{i=1}^{\ell(\sigma^\bullet)}
|\text{Aut}(\sigma^{(i)})|}.$$
Here, $\text{Aut}(\sigma^\bullet)$ is the group
permuting equal subpartitions.
The factor $m(\sigma^\bullet)$ may be interpreted
as counting the number of different ways the disjoint union
can be made.

To write explicitly the $\mathbf{p}^\sigma$ coefficient of
$\exp(\gamma^{\mathbf{p}})$, we introduce the functions 
$$F_{n,m}(t,x) = -\sum_{d=1}^\infty \sum_{s=-1}^\infty C^s_{d}\ \kappa_{s+m} t^{s+m} 
\frac{d^n x^d}{d!}$$
for $n,m \geq 1$.
Then,
\begin{multline*}
|\text{Aut}(\sigma)| \cdot
\Big[ \exp(-\gamma^{\mathbf{p}}) \Big]_{t^rx^d\mathbf{p}^\sigma} =\\
\Big[ \exp(-\gamma(t,x)) 
\cdot
\left(
\sum_{\sigma^\bullet\in \mathcal{S}(\sigma)}
m(\sigma^\bullet)
\prod_{i=1}^{\ell(\sigma^\bullet)}
F_{\ell(\sigma^{(i)}), |\sigma^{(i)}|} \right) 
\Big]_{t^rx^d} \ .
\end{multline*}
The length $\ell(\sigma^{*,\bullet})$ is the number
of unmarked subpartitions.

Let $\sigma^{*,\bullet}$ be a division of $\sigma$ with a 
marked subpartition,
\begin{equation}\label{rrggg}
\sigma = \sigma^* \cup \sigma^{(1)} \cup \sigma^{(2)} \cup \sigma^{(3)}\ldots,
\end{equation}
labelled by the superscript $*$. 
The marked subpartition is permitted to be empty. Let 
$\mathcal{S}^*(\sigma)$ denote the
set of marked divisions of $\sigma$.
Let
$$m(\sigma^{*,\bullet}) = \frac{1}{|\text{Aut}(\sigma^\bullet)|}
\frac{|\text{Aut}(\sigma)|}{|\text{Aut}(\sigma^*)|
\prod_{i=1}^{\ell(\sigma^{*,\bullet})}
|\text{Aut}(\sigma^{(i)})|}.$$
Then, $|\text{Aut}(\sigma)|$ 
times the right side of Theorem 3 may be written as
\begin{multline*}
(-1)^{g+|\sigma|} |\text{Aut}(\sigma)| \cdot
\Big[ \exp(-\gamma(-t,x)) 
\cdot\\
\left(
\sum_{\sigma^{*,\bullet}\in \mathcal{S}^*(\sigma)}
m(\sigma^{*,\bullet})
 \prod_{j=1}^{\ell(\sigma^*)}
\kappa_{\sigma^*_{j}-1} (-t)^{\sigma^*_j-1}
\prod_{i=1}^{\ell(\sigma^{*,\bullet})}
F_{\ell(\sigma^{(i)}), |\sigma^{(i)}|} (-t,x)\right)
\Big]_{t^rx^d}
\end{multline*}

To write Theorem 3 in the simplest form, the following
definition with the Kronecker $\delta$ is useful,
$$m^\pm(\sigma^{*,\bullet}) = 
(1\pm\delta_{0,|\sigma^*|}) \cdot m(\sigma^{*,\bullet}).$$
There are two cases.
If $g\equiv r +|\sigma|$ mod 2, then Theorem 3 is equivalent to
the vanishing of
\begin{multline*}
\Big[ \exp(-\gamma) 
\cdot
\left(
\sum_{\sigma^{*,\bullet}\in \mathcal{S}^*(\sigma)}
m^-(\sigma^{*,\bullet})
 \prod_{j=1}^{\ell(\sigma^*)}
\kappa_{\sigma^*_{j}-1} t^{\sigma^*_j-1}
\prod_{i=1}^{\ell(\sigma^{*,\bullet})}
F_{\ell(\sigma^{(i)}), |\sigma^{(i)}|} \right)
\Big]_{t^rx^d}  .
\end{multline*}
If $g\equiv r +|\sigma|+1$ mod 2, then Theorem 3 is equivalent
to the vanishing of
\begin{multline*}
\Big[ \exp(-\gamma) 
\cdot
\left(
\sum_{\sigma^{*,\bullet}\in \mathcal{S}^*(\sigma)}
m^+(\sigma^{*,\bullet})
 \prod_{j=1}^{\ell(\sigma^*)}
\kappa_{\sigma^*_{j}-1} t^{\sigma^*_j-1}
\prod_{i=1}^{\ell(\sigma^{*,\bullet})}
F_{\ell(\sigma^{(i)}), |\sigma^{(i)}|} \right)
\Big]_{t^rx^d} .
\end{multline*}
In either case, the relations are valid
in the ring $R^*(\M_g)$ only if the condition
$g-2d-1+|\sigma| < r$ holds.

\pagebreak

\noindent V. {\em Further examples.}
\vspace{10pt}

If $\sigma=(k)$ has a single part, then the two cases
of Theorem 3 are the following.
If $g\equiv r+k$ mod 2,  we have
$$\Big[ \exp(-\gamma)\cdot \kappa_{k-1}t^{k-1} \Big]_{t^r x^d}=0\ $$
which is a consequence of Theorem 2.
If $g\equiv r+k+1$ mod 2, we have 
$$\Big[ \exp(-\gamma)\cdot \left(\kappa_{k-1}t^{k-1}
+ 2 F_{1,k}\right) \Big]_{t^r x^d}=0$$

If $\sigma=(k_1k_2)$ has two distinct parts, then the two cases
of Theorem 3 are as follows.
If $g\equiv r+k_1+k_2$ mod 2,  we have
\begin{multline*}\Big[ \exp(-\gamma)\cdot \big(
\kappa_{k_1-1}\kappa_{k_2-1}t^{k_1+k_2-2}\\ +\kappa_{k_1-1}t^{k_1-1} F_{1,k_2}
+\kappa_{k_2-1} t^{k_2-1} F_{1,k_1}\big) 
\Big]_{t^r x^d}=0\ .
\end{multline*}
If $g\equiv r+k_1+k_2+1$ mod 2,  we have
\begin{multline*}\Big[ \exp(-\gamma)\cdot \big(
\kappa_{k_1-1}\kappa_{k_2-1}t^{k_1+k_2-2} +\kappa_{k_1-1} t^{k_1-1}F_{1,k_2}
\\+ \kappa_{k_2-1} t^{k_2-1}F_{1,k_1} 
+ 2 F_{2,k_1+k_2} + 2 F_{1,k_1} F_{1,k_2}
\big) 
\Big]_{t^r x^d}=0\ .
\end{multline*}

In fact, the $g\equiv r+k_1+k_2$ mod 2 equation above is not new.
The genus $g$ and codimension $r_1=r-k_2+1$ case of partition $(k_1)$
yields
$$\Big[ \exp(-\gamma)\cdot \left(\kappa_{k_1-1}t^{k_1-1}
+ 2 F_{1,k_1}\right) \Big]_{t^{r_1} x^d}=0\ .$$
After multiplication with $\kappa_{k_2-1}t^{k_2-1}$, we obtain
$$\Big[ \exp(-\gamma)\cdot \left(\kappa_{k_1-1}\kappa_{k_2-1}t^{k_1+k_2-2}
+ 2 \kappa_{k_2-1}t^{k_2-1} F_{1,k_1}\right) \Big]_{t^{r} x^d}=0\ .$$
Summed with the corresponding equation with $k_1$ and $k_2$
interchanged yields the above 
$g\equiv r+k_1+k_2$ mod 2 case.

\pagebreak

\noindent VI. {\em Expanded form revisited.}
\vspace{10pt}

Consider the partition $\sigma=(k_1k_2\cdots k_\ell)$
with distinct parts. We obtain from Theorem 3, in
the $g\equiv r+|\sigma|$ mod 2 case, the
vanishing of 
\begin{multline*}
\Big[ \exp(-\gamma) 
\cdot
\left(
\sum_{\sigma^{*,\bullet}\in \mathcal{S}^*(\sigma)}
(1-\delta_{0,|\sigma^*|})
 \prod_{j=1}^{\ell(\sigma^*)}
\kappa_{\sigma^*_{j}-1} t^{\sigma^*_j-1}
\prod_{i=1}^{\ell(\sigma^{*,\bullet})}
F_{\ell(\sigma^{(i)}), |\sigma^{(i)}|} \right)
\Big]_{t^rx^d}  
\end{multline*}
since all the factors $m(\sigma^{*,\bullet})$ are
1. In the 
$g\equiv r+|\sigma|+1$ mod 2 case, we obtain the vanishing
of
\begin{multline*}
\Big[ \exp(-\gamma) 
\cdot
\left(
\sum_{\sigma^{*,\bullet}\in \mathcal{S}^*(\sigma)}
(1+\delta_{0,|\sigma^*|})
 \prod_{j=1}^{\ell(\sigma^*)}
\kappa_{\sigma^*_{j}-1} t^{\sigma^*_j-1}
\prod_{i=1}^{\ell(\sigma^{*,\bullet})}
F_{\ell(\sigma^{(i)}), |\sigma^{(i)}|} \right)
\Big]_{t^rx^d}  
\end{multline*}
for the same reason.

\begin{Proposition}
The $g\equiv r+|\sigma|$ mod 2 case is a consequence of
the $g\equiv r'+|\sigma'|+1$ mod 2 cases of smaller
partitions $\sigma'$.
\end{Proposition}

\begin{proof}
The strategy is identical to that employed in the
special cases of the result proven in Section V.
\end{proof}

If $\sigma$ has repeated parts, the relations of
Theorem 3 are obtained by viewing the parts are
distinct and using the above formulas.
For example,
 the two cases
of Theorem 3 
for $\sigma=(k^2)$
are as follows.
If $g\equiv r+2k$ mod 2,  we have
\begin{equation*}\Big[ \exp(-\gamma)\cdot \big(
\kappa_{k-1}\kappa_{k-1}t^{2k-2} +2\kappa_{k-1}t^{k-1} F_{1,k}\big) 
\Big]_{t^r x^d}=0\ .
\end{equation*}
If $g\equiv r+2k+1$ mod 2,  we have
\begin{multline*}\Big[ \exp(-\gamma)\cdot \big(
\kappa_{k-1}\kappa_{k-1}t^{2k-2} +2\kappa_{k-1} t^{k-1}F_{1,k}
\\ 
+ 2 F_{2,2k} + 2 F_{1,k} F_{1,k}
\big) 
\Big]_{t^r x^d}=0\ .
\end{multline*}
The factors occur via repetition of terms in
the formulas for distinct parts.

\pagebreak

\noindent VII. {\em Differential equations.}
\vspace{10pt}

The function $\Phi$ satisfies a basic differential equation obtained
from the series definition,
$$\frac{d}{dx} (\Phi- tx \frac{d}{dx} \Phi) = -\frac{1}{t} \Phi \ .$$
After expanding and dividing by $\Phi$, we find
$$- tx \frac{\Phi_{xx}}{\Phi} - t \frac{\Phi_x}{\Phi}
+ \frac{\Phi_x}{\Phi} = -\frac{1}{t} \ $$
which can be written as
\begin{equation} \label{diffe}
-t^2x \gamma^*_{xx} = t^2 x (\gamma^*_x)^2 + t^2 \gamma^*_x 
-t \gamma^*_x -1 \ 
\end{equation}
where, as before, $\gamma^*=\log(\Phi)$.
Equation  \eqref{diffe} has been studied 
by Ionel in {\em Relations in the tautological ring} \cite{Ion}. We present
here  results of hers which will be useful for us.

To kill the pole and  match the required constant term,
 we will consider
the function
$$\Gamma=-t\left(\sum_{i\geq 1} \frac{B_{2i}}{2i(2i-1)}t^{2i-1} + \gamma^* 
\right)\ .$$ 
The differential equation \eqref{diffe} becomes
$$tx \Gamma_{xx} = x (\Gamma_x)^2 +(1-t) \Gamma_x -1 \ .$$
The differential equation is
easily seen to  uniquely determine $\Gamma$ once
  the initial conditions
$$\Gamma(t,0) = - 
\sum_{i\geq 1} \frac{B_{2i}}{2i(2i-1)} t^{2i}$$
are specified. 
By Ionel's first result,
$$\Gamma_x = \frac{-1+\sqrt{1+4x}}{2x} + \frac{t}{1+4x}
+ \sum_{k=1}^\infty \sum_{j=0}^k t^{k+1} q_{k,j}(-x)^j(1+4x)^{-j-\frac{k}{2}-1}\  $$
where the postive integers $q_{k,j}$ (defined to vanish unless
$k\geq j \geq 0$) are defined via the recursion
$$q_{k,j} = (2k+4j-2)q_{k-1,j-1} + (j+1)q_{k-1,j}
+ \sum_{m=0}^{k-1} \sum_{l=0}^{j-1} q_{m,l} q_{k-1-m,j-1-l}\ $$
from the initial value $q_{0,0}=1$.

\pagebreak

Ionel's second result is obtained by integrating $\Gamma_x$
with respect to $x$. She finds
$$\Gamma = \Gamma(0,x) + \frac{t}{4} \log(1+4x)
-\sum_{k=1}^\infty \sum_{j=0}^k t^{k+1} c_{k,j} (-x)^j (1+4x)^{-j-\frac{k}{2}}\ $$
where the coefficients $c_{k,j}$ are determined by 
$$q_{k,j}=(2k+4j)c_{k,j} +(j+1) c_{k,j+1}$$
for $k\geq 1$ and $k\geq j\geq 0$.

While the derivation of the formula for $\Gamma_x$ is straightforward,
the formula for $\Gamma$ is quite subtle as the intial conditions
(given by the Bernoulli numbers) are used to show the
vanishing of constants of integration. Said differently, the
recusions for $q_{k,j}$ and $c_{k,j}$ must be shown to imply the formula
$$c_{k,0} = \frac{B_k}{k(k-1)}\ .$$
A third result of Ionel's is the determination of the
extremal $c_{k,k}$,
$$\sum_{k=1}^\infty c_{k,k} z^k = \log\left( \sum_{k=1}^\infty
\frac{(6k)!}{(2k)!(3k)!} \left(\frac{z}{72}\right)^k \right)\ .$$

The formula for $\Gamma$ becomes simpler after the following
very natural change of variables,
\begin{equation}\label{varch}
u= \frac{t}{\sqrt{1+4x}} \ \ \ \text{and} \ \ \
y= \frac{-x}{1+4x} \ .
\end{equation}
The change of variables defines a new function
 $$\widehat{\Gamma}(u,y) = \Gamma(t,x) \ .$$
The formula for $\Gamma$ implies
$$\frac{1}{t} \Gamma(u,y) = \frac{1}{t} \Gamma(0,y)
-\frac{1}{4}\log(1+4y)
-\sum_{k=1}^\infty \sum_{j=0}^k  c_{k,j}u^k y^j  \ .
$$

Ionel's fourth result relates coefficients of series
after the change of variables \eqref{varch}.
Given any series
$$P(t,x) \in \mathbb{Q}[[t,x]],$$ let
$\widehat{P}(u,y)$ be the series obtained from
the change of variables \eqref{varch}. Ionel proves
coefficient relation
$$\big[ P(t,x) \big]_{t^r x^d} = (-1)^d \big[
(1+4y)^{\frac{r+2d-2}{2}} \cdot \widehat{P}(u,y) \big]_{u^r y^d}\ .$$

\pagebreak

\noindent VII. {\em Analysis of the relations of Theorem 2}
\vspace{10pt}

We now study in detail the simple relations of Theorem 2,
\begin{equation*}
\big[ \exp(-\gamma) \big]_{t^rx^d} =0 \    \in R^r(\M_g)
\end{equation*}
when $g-2d-1< r$ and 
$g\equiv r+1 \hspace{-5pt} \mod 2$.
Let 
$$\widehat{\gamma}(u,y) = \gamma(t,x)$$
be obtained from the variable change \eqref{varch},
$$\widehat{\gamma}(u,y) = 
\frac{\kappa_0}{4}\log(1+4y)+
\sum_{k=1}^\infty \sum_{j=0}^k \kappa_k c_{k,j}u^k  y^j  \ 
$$
modulo $\kappa_{-1}$ terms which we set to $0$.
Applying Ionel's coefficient result,
\begin{eqnarray*}
\big[ \exp(-\gamma) \big]_{t^rx^d}& = & \big[ 
(1+4y)^{\frac{r+2d-2}{2}} \cdot
\exp(-\widehat{\gamma}) 
\big]_{u^r y^d} \\ & = &
\left[ 
(1+4y)^{\frac{r+2d-2}{2}-\frac{\kappa_0}{4}} \cdot
\exp(-     \sum_{k=1}^\infty \sum_{j=0}^k \kappa_k c_{k,j}u^k  y^j    ) 
\right]_{u^r y^d} \\
& = & 
\left[ 
(1+4y)^{\frac{r-g+2d-1}{2}} \cdot
\exp(-     \sum_{k=1}^\infty \sum_{j=0}^k \kappa_k c_{k,j}u^k y^j    ) 
\right]_{u^r y^d} \ .
\end{eqnarray*}
In the last line, the substitution $\kappa_0=2g-2$
has been made.

Consider first the exponent of $1+4y$.
By the assumptions on $g$ and $r$ in Theorem 2,
$$\frac{r-g+2d-1}{2}\geq 0$$ 
and the fraction is integral. Hence, the $y$ degree of
the prefactor
$$(1+4y)^{\frac{r-g+2d-1}{2}}$$
is exactly $\frac{r-g+2d-1}{2}$.
The $y$ degree of the exponential factor is bounded
from above by the $u$ degree. We conclude
$$\left[ 
(1+4y)^{\frac{r-g+2d-1}{2}} \cdot
\exp(-     \sum_{k=1}^\infty \sum_{j=0}^k \kappa_k c_{k,j}u^k  y^j    ) 
\right]_{u^r y^d} =0$$
is the  {\em trivial} relation unless
$$r \geq d - {\frac{r-g+2d-1}{2}} = -\frac{r}{2} +\frac{g+1}{2} \ .$$
Rewriting the inequality, we obtain
$3r \geq g+1$
which is equivalent to $r > \lfloor \frac{g}{3} \rfloor$.
The conclusion is in agreement with the proven freeness of $R^*(\M_g)$
up to (and including) degree $\lfloor \frac{g}{3} \rfloor$.

A similar connection between Theorem 2 and Ionel's relations in
\cite{Ion} has also been found by Shengmao Zhu \cite{zhu}. 

\vspace{+10pt}
\noindent VIII. {\em Analysis of the relations of Theorem 3}
\vspace{10pt}

For the relations of Theorem 3, we will require additional
notation.  To start, 
let
$$\gamma^c(u,y) =      \sum_{k=1}^\infty \sum_{j=0}^k \kappa_k c_{k,j}u^k  y^j  \ .$$By Ionel's second result,
$$\frac{1}{t}\Gamma = \frac{1}{t} \Gamma(0,x) + \frac{1}{4} \log(1+4x)
-\sum_{k=1}^\infty \sum_{j=0}^k t^{k} c_{k,j} 
(-x)^j (1+4x)^{-j-\frac{k}{2}}\ .$$
Let $c_{k,j}^0 = c_{k,j}$.
We  define the constants $c_{k,j}^n$ for $n\geq 1$ by
\begin{multline*}
\left( x \frac{d}{dx}\right)^n \frac{1}{t}\Gamma 
=\left( x \frac{d}{dx}\right)^{n-1} \left( \frac{-1}{2t} + \frac{1}{2t}\sqrt{1+4x}\right)\\ -
\sum_{k=0}^\infty \sum_{j=0}^{k+n} t^{k} 
c^n_{k,j} (-x)^j (1+4x)^{-j-\frac{k}{2}} 
\ . 
\end{multline*}

\begin{Lemma} For $n>0$, there are constants $b^n_j$ satisfying 
$$
\left( x \frac{d}{dx}\right)^{n-1} \left(\frac{1}{2t}\sqrt{1+4x}\right)
= \sum_{j=0}^{n-1} b^n_j u^{-1} y^j \ .
$$
Moreover, $b^n_{n-1} = -2^{n-2}\cdot(2n-5)!!$ where
$(-1)!!=1$ and $(-3)!!=-1$.
\end{Lemma}
\begin{proof}
The result is obtained by simple induction.
The negative evaluations $(-1)!!=1$ and $(-3)!!=-1$ arise from the
$\Gamma$-regularization. 
\end{proof}

\begin{Lemma} For $n>0$, we have $c_{0,n}^n = 4^{n-1}(n-1)!$.
\end{Lemma}

\begin{Lemma} For $n>0$ and $k>0$, we have 
$$c_{k,k+n}^n =  (6k)(6k+4)\cdots (6k+4(n-1))\ 
c_{k,k}.$$
\end{Lemma}

Consider next the full set of equations given by Theorem 3
in the expanded form of Section VI.
The function $F_{n,m}$ may be rewritten as
\begin{eqnarray*}
F_{n,m}(t,x) & =& 
- \sum_{d=1}^\infty \sum_{s=-1}^\infty C^s_{d}\ \kappa_{s+m} t^{s+m} 
\frac{d^n x^d}{d!} \\
& = & -t^m \left(x\frac{d}{dx}\right)^n 
\sum_{d=1}^\infty \sum_{s=-1}^\infty C^s_{d}\ \kappa_{s+m} t^{s} 
\frac{x^d}{d!}.
\end{eqnarray*}
We may write the result in terms of the constants $b^n_j$ and $c^n_{k,j}$,
\begin{multline*}
t^{-(m-n)}F_{n,m} = -\delta_{n,1}\frac{\kappa_{m-1}}{2} \\ 
+
(1+4y)^{-\frac{n}{2}} \Big( \sum_{j=0}^{n-1}\kappa_{m-1}b_j^n u^{n-1} y^j
 -
\sum_{k=0}^\infty \sum_{j=0}^{k+n}  \kappa_{k+m}  c^n_{k,j} u^{k+n} y^j \Big)
\end{multline*}
Define the functions
$G_{n,m}(u,y)$ by
$$G_{n,m}(u,y) = \sum_{j=0}^{n-1}\kappa_{m-1}b_j^n u^{n-1} y^j
 -
\sum_{k=0}^\infty \sum_{j=0}^{k+n}  \kappa_{k+m}  c^n_{k,j} u^{k+n} y^j \ .$$

Let $\sigma=(1^{a_1}2^{a_2}3^{a_3} \ldots)$ be a partition
of length $\ell(\sigma)$ and size $|\sigma|$.
We assume the parity condition 
\begin{equation}\label{par2}
g\equiv  r + |\sigma|+1\ .
\end{equation}
Let $G_\sigma^\pm(u,y)$ be the following function associated to $\sigma$,
$$G_\sigma^\pm(u,y) =
\sum_{\sigma^{\bullet}\in \mathcal{S}(\sigma)}
\prod_{i=1}^{\ell(\sigma^{\bullet})}
\left(G_{\ell(\sigma^{(i)}), |\sigma^{(i)}|} \pm
\frac{\delta_{\ell(\sigma^{(i)}),1}}{2} 
\sqrt{1+4y}\ 
\kappa_{|\sigma^{(i)}|-1}\right)\ .$$
The relations of Theorem 3 written in the variables $u$ and $y$ is
\begin{equation*}
\Big[ (1+4y)^{\frac{r-|\sigma|-g+2d-1}{2}}
\exp(-\gamma^c) 
\left( G_\sigma^+ + G_\sigma^-\right)
\Big]_{u^{r-|\sigma|+\ell(\sigma)}y^d}  = 0 
\end{equation*}

In fact, the relations of Theorem 3 can be written in a much more
efficient form when the strategy of Proposition 2 is used
to take out lower equations.

\begin{Theorem} In $R^r(\M_g)$, the relation 
\begin{equation*}
\Big[ (1+4y)^{\frac{r-|\sigma|-g+2d-1}{2}}
\exp\left(-\gamma^c +\sum_{\sigma \neq \emptyset} G_{\ell(\sigma),|\sigma|} 
\frac{\mathbf{p}^\sigma}{|\text{\em Aut}(\sigma)|} \right) 
\Big]_{u^{r-|\sigma|+\ell(\sigma)}y^d\mathbf{p}^\sigma}  = 0 
\end{equation*}
holds when 
$g-2d-1 +|\sigma| < r$ and $g\equiv  r + |\sigma|+1$ mod 2.
\end{Theorem}

\pagebreak

Consider the  exponent of $1+4y$.
By the inequality and the  parity condition \eqref{par2},
$$\frac{r-|\sigma|-g+2d-1}{2}\geq 0$$ 
and the fraction is integral. Hence, the $y$ degree of
the prefactor
$$(1+4y)^{\frac{r-|\sigma|-g+2d-1}{2}}$$
is exactly $\frac{r-|\sigma|-g+2d-1}{2}$.
The $y$ degree of the exponential factor is bounded
from above by the $u$ degree. We conclude
the relation of Theorem 4 is
 {\em trivial} unless
$$r-|\sigma|+\ell(\sigma) \geq d - {\frac{r-|\sigma|-g+2d-1}{2}} 
= -\frac{r-|\sigma|}{2} +\frac{g+1}{2} \ .$$
Rewriting the inequality, we obtain
$$3r \geq g+1 + 3|\sigma|-2\ell(\sigma)$$
which is consistent with the proven freeness of $R^*(\M_g)$
up to (and including) degree $\lfloor \frac{g}{3} \rfloor$.



\vspace{60pt}

$$ $$
$$ $$
$$ $$

\pagebreak

\vspace{10pt}
\noindent X. {\em Another form}
\vspace{10pt}

A subset of the equations of Theorem 4 admits an especially simple
description.
Consider the function
\begin{multline*}
H_{n,m}(u) =  2^{n-2}(2n-5)!! \ \kappa_{m-1} u^{n-1} 
 + 4^{n-1}(n-1)!\ \kappa_m u^n \\
+ \sum_{k=1}^\infty  (6k)(6k+4)\cdots (6k+4(n-1)) c_{k,k}\   
\kappa_{k+m}
u^{k+n} \ .
\end{multline*}

\begin{Proposition} In $R^r(\M_g)$, the relation 
\begin{equation*}
\Big[ 
\exp\left(-\sum_{k=1}^\infty c_{k,k} \kappa_k u^k 
-\sum_{\sigma \neq \emptyset} H_{\ell(\sigma),|\sigma|} 
\frac{\mathbf{p}^\sigma}{|\text{\em Aut}(\sigma)|} \right) 
\Big]_{u^{r-|\sigma|+\ell(\sigma)}\mathbf{p}^\sigma}  = 0 
\end{equation*}
holds when 
$3r \geq g+1 + 3|\sigma|-2\ell(\sigma)$
 and $g\equiv  r + |\sigma|+1$ mod 2.
\end{Proposition}

The main advantage of Proposition 3 is the dependence
on only the function
\begin{equation}\label{njjk}
\sum_{k=1}^\infty c_{k,k} z^k = \log\left( \sum_{k=1}^\infty
\frac{(6k)!}{(2k)!(3k)!} \left(\frac{z}{72}\right)^k \right)\ .
\end{equation}
Proposition 3 only provides finitely many relations
for fixed $g$ and $r$.
In Theorem 5 of Section C.I below, a more
elegant set of relations in $R^r(M_g)$ conjectured by
Faber-Zagier is presented.
In fact, we show Proposition 3 is equivalent to the Faber-Zagier
conjecture. 

\vspace{60pt}

$$ $$
$$ $$
$$ $$

\pagebreak

\vspace{10pt}
\noindent {\bf C. The conjecture of Faber-Zagier}
\vspace{10pt}

\vspace{10pt}
\noindent I. {\em The function $\Psi$}
\vspace{10pt}

A third set of relations is defined as follows
Let 
$$\mathbf{p} = \{\ p_1,p_3,p_4,p_6,p_7,p_9,p_{10}, \ldots\ \}$$
be a variable set indexed by integers not congruent
to $2$ mod 3.
Let
\begin{multline*}
\Psi(t,\mathbf{p}) =
(1+tp_3+t^2p_6+t^3p_9+\ldots) \sum_{i=0}^\infty \frac{(6i)!}{(3i)!(2i)!} t^i
\\ +(p_1+tp_4+t^2p_7+\ldots) 
\sum_{i=0}^\infty \frac{(6i)!}{(3i)!(2i)!} \frac{6i+1}{6i-1} t^i
\end{multline*}
Define the constants $C^r(\sigma)$ by the formula
$$\log(\Psi)= 
\sum_{\sigma}
\sum_{r=0}^\infty C^r(\sigma)\ t^r 
\mathbf{p}^\sigma
\ . $$
Here and below, $\sigma$ denotes a partition which avoids all 
 parts congruent to 2 mod 3.
Let 
$$\gamma= 
\sum_{\sigma}
 \sum_{r=0}^\infty C^r(\sigma)
\ \kappa_r t^r 
\mathbf{p}^\sigma
\ .
$$

Our main result, starting from the stable quotient
relations, is the following final form.

\begin{Theorem} 
In $R^r(\M_g)$, the relation
$$
\big[ \exp(-\gamma) \big]_{t^r \mathbf{p}^\sigma}  = 0$$
holds when
$g-1+|\sigma|< 3r$ and
$g\equiv r+|\sigma|+1 \mod 2$.
\end{Theorem}

The relations of Theorem 5 were conjectured earlier by Faber and Zagier
from data and a study of the
Gorenstein quotient of $R^*(\M_g)$.
To the best of our knowledge, a relation in $R^*(\M_g)$ which is
not in the span of the relations of Theorem 5 has not yet been found.
In particular, all relations obtained from Theorem 1 to date
are in the span of Theorem 5 (and conversely).
It is very reasonable to expect the spans of the relations in
Theorem 1 and Theorem 5 exactly coincide.
Whether Theorem 5 exhausts all relations in $R^*(\M_g)$ is
a very interesting question. 

Theorem 5 is much more efficient than Theorem 1 for several
reasons. Theorem 5 only provides finitely many relations
in $R^r(\M_g)$ for fixed $g$ and $r$, and thus may be 
calculated completely. 
When the relations yield a Gorenstein
ring with socle in $R^{g-2}(\M_g)$, no further relations are possible.
However, the relations of Theorem 5 do not always yield such a
Gorenstein ring (failing first in genus 24 as checked by Faber).
For $g<24$, Faber's calculations show
Theorem 5 does provide all relations in $R^*(\M_g)$.
For higher genus $g\geq 24$, either Theorem 5 fails to
provide all the relations in $R^*(\M_g)$ {\em or}
$R^*(\M_g)$ is not Gorenstein.

\vspace{10pt}
\noindent II. {\em Connection to the stable quotient relations}
\vspace{10pt}

Theorem 5 is derived from Proposition 3.
In fact, Proposition 3 is equivalent to Theorem 5. The
derivation is obtained by a triangular transformation
among distinguished generators.
A certain amount of differential algebra is required.

Consider the relation obtained from the
partition $\sigma=(1)$ in Proposition 3 and the Conjecture.
 For convenience, let
\begin{eqnarray*}
A(z)& = &\sum_{i=0}^\infty \frac{(6i)!}{(3i)!(2i)!} 
\left(\frac{z}{72}\right)^i\ , \\
B(z)& = & \sum_{i=0}^\infty \frac{(6i)!}{(3i)!(2i)!} \frac{6i+1}{6i-1} 
\left(\frac{z}{72}\right)^i
\end{eqnarray*}
The conjectures predict
no room for different relations in $R^*(\M_g)$ for $\sigma=(1)$, so
we must have
$$-\frac{1}{2}+ z + 6 z\left(z\frac{d}{dz}\right) \log(A) $$
proportional to $B/A$. We find
the equation
$$-\frac{1}{2}+ z + 6 z\left(z\frac{d}{dz}\right) 
\log(A)= \frac{1}{2}\frac{B}{A}\ $$
holds.
Equivalently,
$$-\frac{1}{2}A + zA + 6 z^2\ \frac{dA}{dz} = \frac{1}{2}B\ $$

More interesting is the partition $\sigma=(11)$.
Here we predict, once the definitions are unwound, that
$$z+ 4z^2 + 36 z^2\left( z\frac{d}{dz}\right)^2 \log(A) 
+ 24 z^2\left(z\frac{d}{dz}\right)\log A$$
is a linear combination of $1$ and $B^2/A^2$.
 We find
the equation
$$z+ 4z^2 + 36 z^2\left( z\frac{d}{dz}\right)^2 \log(A) 
+ 24 z^2\left(z\frac{d}{dz}\right)\log A= 
\frac{1}{4} - \frac{1}{4} \frac{B^2}{A^2}
$$
holds.

In fact, the main hypergeometric differential equation satisfied
by the function $A$ is
$$36 z^2 \frac{d^2}{dz^2} A + (72z-6) \frac{d}{dz} A + 5 A = 0 \ .$$
In Section D below, further details describing the use of such differential
equations 
to prove Theorem 5 from
Proposition 3 are presented.

\vspace{10pt}
\noindent III. {\em Functions}
\vspace{10pt}

While the functions $A(z)$ and $B(z)$ of Section II have radius
of convergence 0, an additional double factorial in the
denominator yields convergent classical series,

\begin{eqnarray*}
\frac{3}{2t}\sin \left(\frac{2}{3} \sin^{-1}(t)\right)& 
= &\sum_{i=0}^\infty \frac{(6i)!}{(3i)!(2i)!(2i+1)!!} 
\left(\frac{t^2}{216}\right)^i\ , \\
-\frac{3}{4t}\sin \left(\frac{4}{3} \sin^{-1}(t)\right)
& = & \sum_{i=0}^\infty \frac{(6i)!}{(3i)!(2i)!(2i+1)!!} \frac{6i+1}{6i-1} 
\left(\frac{t^2}{216}\right)^i\ .
\end{eqnarray*}

\vspace{10pt}
\noindent IV. {\em Remarks}
\vspace{10pt}

Stable quotients relations in $R^*(\M_g)$ have several advantages
over other geometric constructions. We have already seen here
the possibility of exact evaluation. Another advantage we have not
explored in these lectures is the extension of the stable
quotients relations over $\overline{\M}_g$. The boundary
terms of the stable quotients relations are tautological.
A study of the relations among the $\kappa$ classes in the
tautological ring of the moduli space of curves of compact
type $\M_{g,n}^c$ has been undertaken in \cite{kap1,kap2}. For example,
the Gorenstein predictions for $\kappa$ classes are proven there
if $n\geq 1$.
But even for $\M_g$, the extension of the stable quotients
relations over $\overline{\M}_g$ has significant consequences.
Since we know the stable quotients relations are all
relations in $R^*(\M_g)$ for $g<24$, the following result holds.

\begin{Proposition} For $g<24$, we have a right exact sequence
$$R^*(\partial \overline{\M}_g) \rightarrow R^*(\overline{\M}_g)
\rightarrow R^*(\M_g) \rightarrow  0 \ . $$
\end{Proposition}

Speculations about such right exactness for tautological rings
were advanced in \cite{FPrel}.

\vspace{20pt}
\noindent {\bf D. The equivalence}
\vspace{10pt}

\vspace{10pt}
\noindent I. {\em Notation}
\vspace{10pt}

The relations in Theorem 5 and Proposition 3 have a similar flavor.
 We start with formal series related  to
\[
A(z) = \sum_{i=0}^\infty \frac{(6i)!}{(3i)!(2i)!}\left(\frac{z}{72}\right)^i,
\]
we insert classes $\kappa_r$, we 
exponentiate, and we extract  coefficients to obtain 
relations among the $\kappa$ classes. 
In order to make the similarities clearer, 
we will introduce additional notation. 

If $F$ is a formal power series in $z$,
\[
F = \sum_{r=0}^\infty c_rz^r
\]
with coefficients in a ring $R$, let
\[
\{F\}_\kappa = \sum_{r=0}^\infty c_r\kappa_rz^r
\]
be the series with $\kappa$-classes inserted.

Let $A$ be as above, and let $B$ be the function defined in C.II.
Let
$$C = \frac{B}{A}\ ,$$ 
and let 
$$E = \exp(-\{\log(A)\}_\kappa) = \exp\left(-\sum_{k=1}^\infty c_{k,k}\kappa_kz^k\right).$$
We will rewrite the relations of Theorem 5 and Proposition 3 in terms of $C$ and $E$.
The equivalence between the two will rely on properties of the differential equations 
satisfied by $C$.

\vspace{10pt}
\noindent II. {\em Rewriting the relations}
\vspace{10pt}

The relations of Theorem 5 conjectured by Faber-Zagier
are straightforward to rewrite using the above notation:
\begin{multline}\label{FZ0}
\Bigg[E\cdot\exp\Big(-\Big\{\log(1+p_3z+p_6z^2+\cdots\\
+C(p_1+p_4z+p_7z^2+\cdots))\Big\}_\kappa\Big)\Bigg]_{z^rp^\sigma} = 0
\end{multline}
for $3r \ge g+|\sigma|+1$ and $3r\equiv g+|\sigma|+1$ mod $2$.
We call the above relations $\FZ$.

The stable quotient relations of Proposition 3 are a bit more complicated 
to rewrite in terms of $C$ and $E$. Let
\begin{multline*}
2^{-n}C_n = 2^{n-2}(2n-5)!!z^{n-1} + 4^{n-1}(n-1)!z^n \\
+ \sum_{k=1}^\infty (6k)(6k+4)\cdots(6k+4(n-1))c_{k,k}z^{k+n}.
\end{multline*}
We see
$$H_{n,m}(z) = 2^{-n}z^{n-m}\{z^{m-n}C_n\}_\kappa.$$
 The series $C_n$ satisfy
\[
C_1 = C, \ \ \ \ C_{i+1} = \left(12z^2\frac{d}{dz}-4iz\right)C_i.
\]
Since $C$ satisfies the differential equation 
$$12z^2\frac{dC}{dz} = 1 + 4zC - C^2,$$
 each $C_n$ can be expressed as a polynomial in $C$ and $z$:
\[
C_1 = C, \ \ C_2 = 1-C^2,\ \  C_3 = -8z-2C+2C^3, \ldots, \ .
\]

Proposition 3 can then
be rewritten as follows (after an appropriate change of variables):
\begin{equation}\label{SQ}
\left[E\cdot\exp\left(-\sum_{\sigma\ne\emptyset}\{z^{|\sigma|-\ell(\sigma)}C_{\ell(\sigma)}\}_\kappa\frac{p^\sigma}{|\Aut(\sigma)|}\right)\right]_{z^rp^\sigma} = 0
\end{equation}
for $3r \ge g+3|\sigma|-2\ell(\sigma)+1$ and $3r \equiv g+3|\sigma|-2\ell(\sigma)+1$ mod $2$.
We call the stable quotients relations $\SQ$.

The $\FZ$ and $\SQ$
relations  now look much more similar, but the relations in (\ref{FZ0}) are 
indexed by partitions with no parts of size $2$ mod $3$ and
satisfy a slightly different inequality. 
The indexing differences can be erased 
by noting the variables $p_{3k}$ are actually not necessary in (\ref{FZ0}) 
if we are just interested in the \emph{ideal} generated by a set of relations 
(rather than the linear span). If we remove the variables $p_{3k}$
and reindex the others, 
we obtain the following equivalent form of the $\FZ$ relations:
\begin{equation}\label{FZ}
\Big[E\cdot\exp\big(-\big\{\log(1+C(p_1+p_2z+p_3z^2+\cdots))\big\}_\kappa\big)\Big]_{z^rp^\sigma} = 0
\end{equation}
for $3r \ge g+3|\sigma|-2\ell(\sigma)+1$ and $3r \equiv g+3|\sigma|-2\ell(\sigma)+1$ mod $2$.

\vspace{10pt}
\noindent II. {\em Comparing the relations}
\vspace{10pt}

We now explain how to write the $\SQ$ relations (\ref{SQ}) as linear combinations of the 
$\FZ$ relations (\ref{FZ}) with coefficients in $\Q[\kappa_0,\kappa_1,\kappa_2,\ldots]$.
In fact,
the associated matrix will be triangular with diagonal entries equal to $1$.  

We start with further notation.
For a partition $\sigma$, let
\[
\FZ_\sigma = \left[\exp\left(-\left\{\log(1+C(p_1+p_2z+p_3z^2+\cdots))\right\}_\kappa\right)\right]_{p^\sigma}
\]
and
\[
\SQ_\sigma = \left[\exp\left(-\sum_{\sigma\ne\emptyset}\{z^{|\sigma|-\ell(\sigma)}C_{\ell(\sigma)}\}_\kappa\frac{p^\sigma}{|\Aut(\sigma)|}\right)\right]_{p^\sigma}
\]
be power series in $z$ with coefficients that are polynomials in the $\kappa$ classes. 
The relations themselves are given by $[E\cdot\SQ_\sigma]_{z^r}$ and $[E\cdot\FZ_\sigma]_{z^r}$.

For each $\sigma$, we can write $\SQ_\sigma$ in terms of the $\FZ_\sigma$. For example,
\begin{align*}
\SQ_{(111)} &= -\frac{1}{6}\{C_3\}_\kappa + \frac{1}{2}\{C_2\}_\kappa\{C_1\}_\kappa -\frac{1}{6}\{C_1\}_\kappa^3 \\
&= \frac{4}{3}\kappa_1z + \frac{1}{3}\{C\}_\kappa - \frac{1}{3}\{C^3\}_\kappa + \frac{1}{2}(\kappa_0 - \{C^2\}_\kappa)\{C\}_\kappa - \frac{1}{6}\{C\}_\kappa^3 \\
&= \left(\frac{4}{3}\kappa_1z\right) + \left(\left(\frac{1}{3} + \frac{\kappa_0}{2}\right)\{C\}_\kappa\right)\\
& \ \ \ \ \ \ \ \ \ \ \ \  + \left(-\frac{1}{3}\{C^3\}_\kappa -\frac{1}{2}\{C^2\}_\kappa\{C\}_\kappa -\frac{1}{6}\{C\}_\kappa^3\right) \\
&= \frac{4}{3}\kappa_1z\FZ_\emptyset + \left(-\frac{1}{3} - \frac{\kappa_0}{2}\right)\FZ_{(1)} + \FZ_{(111)}.
\end{align*}
We then 
obtain a corresponding linear relation between the relations themselves:
\[
[E\cdot\SQ_{(111)}]_{z^r} = \frac{4}{3}\kappa_1[E\cdot\FZ_{\emptyset}]_{z^{r-1}} + \left(-\frac{1}{3} - \frac{\kappa_0}{2}\right)[E\cdot\FZ_{(1)}]_{z^r} + [E\cdot\FZ_{(111)}]_{z^r}.
\]

Constructing such linear combinations in general
is not hard. When expanded in terms of $C$ as in the above example, 
$\FZ_\sigma$ always contains exactly one term of the form 
\begin{equation}\label{g66h}
\{z^{a_1}C\}_\kappa\{z^{a_2}C\}_\kappa\cdots\{z^{a_m}C\}_\kappa\ . 
\end{equation}
All the other terms involve higher powers of $C$. 
If we expand $\SQ_\sigma$ in terms of $C$, 
we can look at the terms of the form \eqref{g66h} which appear to determine 
how to write the $\SQ_\sigma$ as a linear combination of the $\FZ_{\widehat{\sigma}}$.

We must check the terms involving higher powers of $C$ also match up.
The matching 
 amounts to proving an identity between the coefficients of $C_i$ when expressed a 
polynomial in $C$.
Define polynomials $f_{ij}\in\Z[z]$ by
\[
C_i = \sum_{j=0}^if_{ij}C^j,
\]
and let
\[
f = 1+\sum_{i,j\ge 1}\frac{(-1)^{j-1}f_{ij}}{i!(j-1)!}x^iy^j.
\]
\begin{Lemma}
There exists a power series $g\in\Q[z][[x]]$ such that $f = e^{yg}$.
\end{Lemma}

The Lemma (which can be proven in straightforward fashion using the differential equation satisfied by $f$) 
is the precise consistency statement needed to express the $\SQ$ relations as linear combinations of the $\FZ$ relations. 
The associated matrix is triangular with respect to the partial ordering of partitions by size, 
and the diagonal entries are easily computed to be equal to $1$. Hence, the matrix
 is invertible. 
We conclude the $\SQ$ relations are equivalent to the $\FZ$ relations.


\begin{thebibliography}{[(III)]}
\bibitem {Faber}
C.~Faber, 
\newblock{\em A conjectural description of the tautological ring of the moduli s
pace of curves}, Moduli of curves and abelian varieties,  109--129, Aspects Math., 
Vieweg, Braunschweig, 1999.


\bibitem{FPrel}
C. Faber and R. Pandharipande, {\em Relative maps and tautological
classes}, JEMS {\bf 7} (2005), 13--49.


\bibitem{Ion} E. Ionel, {\em Relations in the tautological ring
of $\M_g$}, Duke Math. J. {\bf 129} (2005), 157--186.


\bibitem{MOP} 
A.~Marian, D. ~Oprea, and R.~Pandharipande, \newblock{\em 
The moduli space of stable quotients}, arXiv:0904.2992.

\bibitem{kap1} R. Pandharipande, {\em The $\kappa$ ring of the
moduli of curves of compact type: I}, arXiv:0906.2657.

\bibitem{kap2} R. Pandharipande, {\em The $\kappa$ ring of the
moduli of curves of compact type: II}, arXiv:0906.2658.

\bibitem{PP} R. Pandharipande and A. Pixton, {\em Stable quotients
and the Faber-Zagier relations}, in preparation.

\bibitem{zhu} S. Zhu, {\em Note on the relations in the
tautological ring of $\M_g$}, preprint 2010.


\end{thebibliography}
\end{document}